\documentclass[hidelinks,onefignum,onetabnum]{siamart251216}

\usepackage{lipsum}
\usepackage{amsfonts}
\usepackage{graphicx}
\usepackage{epstopdf}
\ifpdf
  \DeclareGraphicsExtensions{.eps,.pdf,.png,.jpg}
\else
  \DeclareGraphicsExtensions{.eps}
\fi

\newsiamremark{remark}{Remark}
\newsiamremark{hypothesis}{Hypothesis}
\crefname{hypothesis}{Hypothesis}{Hypotheses}
\newsiamthm{claim}{Claim}
\newsiamremark{fact}{Fact}
\crefname{fact}{Fact}{Facts}

\headers{Gene regulatory networks from RNA velocity}{L. Meng and S. Wang}

\title{Decoding gene regulatory networks from single-cell RNA velocity\thanks{Submitted on August 30, 2026.
\funding{This work was funded by Tianyuan Mathematical Center in Northeast China and Yanbian University.}}}

\author{Lingqi Meng\thanks{Department of Mathematics, Yanbian University, Yanbian Korean Autonomous Prefecture, Jilin 133002, China
  (\email{lingqime@ybu.edu.cn}, \url{https://lingqime.github.io/}).}
\and Shiruo Wang\thanks{Department of Mathematics, State University of New York at Buffalo, Buffalo, New York 14260, United States 
  (\email{shiruomath@gmail.com}).}
}

\usepackage{amsopn}

\ifpdf
\hypersetup{
  pdftitle={Decoding gene regulatory networks from single-cell RNA velocity},
  pdfauthor={L. Meng and S. Wang}
}
\fi

\usepackage{algorithm}
\usepackage{algpseudocode}

\begin{document}

\maketitle

\begin{abstract}
We formulate gene regulatory network reconstruction from RNA velocity as a sparse dynamical inverse problem. We show that control-only data can be structurally nonidentifying and characterize excitation conditions under which controlled perturbations restore identifiability by generating complementary regulator trajectories. To enable stable reconstruction from noisy data, we develop an integral sparse estimator that avoids numerical differentiation and derive recovery bounds separating stochastic error from systematic contributions due to latent-time uncertainty, kinetic-parameter error, numerical quadrature, and model misspecification. Synthetic experiments illustrate perturbation-assisted identifiability, improved conditioning, and the robustness of integral reconstruction. Applied to perturbation-resolved RPE1 RNA-velocity data, the framework yields an empirically full-rank design whose conditioning improves with perturbational diversity and a reconstructed network core stable under perturbation subsampling. Held-out evaluation further shows that identifiability and reconstruction stability do not imply uniform predictive improvement. These results connect perturbational excitation, identifiability, and stable sparse recovery in regulatory dynamical systems.
\end{abstract}



\section{Introduction}

Single-cell RNA sequencing (scRNA-seq) provides snapshot measurements of heterogeneous cellular populations but does not directly observe their temporal evolution \cite{tang2009mrna, macosko2015highly, zheng2017massively}. RNA velocity augments these snapshots with directional information inferred from unspliced and spliced RNA abundances \cite{la2018rna}, with subsequent methods incorporating dynamical transcription models, latent biological time, and gene-specific kinetic parameters \cite{bergen2020generalizing, gorin2022rna}. These developments raise the natural inverse problem of determining to what extent the underlying gene regulatory network (GRN) can be recovered from RNA-velocity observations.

GRN reconstruction has been studied using statistical, causal, and dynamical-system approaches \cite{huynh2010inferring, marbach2012wisdom}. More recently, RNA-velocity-based methods have combined estimated dynamics with sparse regression, ordinary differential equations, and causal or mechanistic models to infer regulatory interactions \cite{aubin2020gene, qiu2020inferring, singh2024causal}. These methods primarily concern estimation of regulatory structure. A logically prior question is whether the regulatory parameters are identifiable from the available observations. Distinct parameter values may generate indistinguishable observations even in the absence of noise, in which case no reconstruction method can uniquely recover the underlying network.

Identifiability and the role of informative inputs have long been studied in dynamical systems and system identification \cite{walter1997identification, hermann1977nonlinear, ljung1998system, villaverde2019observability}. Related work on biological networks has shown that experimental perturbations can improve parameter identifiability by supplying additional dynamical information \cite{zak2001simulation, gross2020identifiability}. The RNA-velocity setting, however, differs from standard time-series identification because observations are destructive snapshots indexed by latent biological time, transcription and splicing are coupled, and multiple perturbation conditions generate distinct regulator trajectories. Existing results therefore do not directly characterize which regulatory directions are invisible from control RNA-velocity observations, when perturbations remove these ambiguities, or how errors in reconstructed trajectories and kinetic parameters affect network recovery. Perturbation-based single-cell experiments provide a natural source of excitation for this inverse problem \cite{dixit2016perturb, adamson2016multiplexed}. By altering regulator activity, these perturbations can generate complementary trajectories that reveal parameter directions absent from the control condition.

We formulate GRN reconstruction from RNA velocity as a sparse dynamical inverse problem arising from coupled transcription--splicing dynamics. We show constructively that control-only observations can be structurally nonidentifying because distinct regulatory parameters may induce identical transcriptional forcing along the observed regulator trajectory. We then characterize how controlled perturbations restore identifiability by enlarging the observable feature space. For the linear transcription model, exact identifiability is equivalent to positive definiteness of an aggregated population information matrix, while for nonlinear regulation an analogous sensitivity condition gives local first-order identifiability. These results also distinguish structural identifiability from numerical conditioning.

Reconstruction from noisy snapshot data presents a separate difficulty because direct differential-equation fitting requires numerical differentiation of estimated trajectories. Integral and weak-form approaches have been developed in dynamical-system identification to mitigate this instability \cite{schaeffer2017sparse, messenger2021weak}. Motivated by this idea, we derive an integral formulation of the transcription dynamics that avoids numerical differentiation while preserving linear dependence on the unknown regulatory coefficients. The resulting problem is a sparse linear inverse problem. We derive finite-sample recovery guarantees, conditional on the estimated trajectories and kinetic parameters, that separate stochastic error from persistent contributions associated with latent-time uncertainty, kinetic-parameter error, numerical quadrature, and model misspecification.

We illustrate the theory through synthetic experiments and apply the reconstruction framework to perturbation-resolved RPE1 RNA-velocity data. The experiments examine the relation between perturbational diversity, identifiability and conditioning, reconstruction stability, and held-out predictive generalization.

The principal contributions of this paper are as follows.

\begin{enumerate}
\item We characterize control-only nonidentifiability geometrically by identifying regulatory parameter directions that are observationally invisible along the control trajectory.

\item We establish perturbation-assisted identifiability conditions. For the linear transcription model, positive definiteness of an aggregated information matrix is necessary and sufficient for exact identifiability; for nonlinear regulation, a sensitivity condition yields local first-order identifiability.

\item We develop an integral sparse reconstruction method that avoids numerical differentiation and derive finite-sample recovery guarantees accounting for stochastic error, latent-time uncertainty, kinetic-parameter error, numerical quadrature, and model misspecification.

\item We validate the framework using synthetic experiments and perturbation-resolved RPE1 RNA-velocity data, examining empirical identifiability, conditioning, reconstruction stability, and held-out predictive generalization.
\end{enumerate}

The remainder of the paper is organized as follows. Section~\ref{Mathematical formulation} introduces the regulatory RNA-velocity model and observation framework. Section~\ref{Nonidentifiability of the control-only inverse problem} establishes control-only nonidentifiability, and Section~\ref{Intervention-assisted identifiability} develops perturbation-assisted identifiability theory. Sections~\ref{Integral sparse reconstruction} and~\ref{Finite-sample recovery theory} introduce the integral sparse estimator and its finite-sample recovery theory, respectively. Section~\ref{Numerical experiments} presents the synthetic experiments and the perturbation-resolved RPE1 real-data application.

\section{Mathematical formulation}
\label{Mathematical formulation}

\subsection{Gene-regulatory RNA-velocity model}

Consider $G$ genes and a prescribed subset $\mathcal R\subset\{1,\ldots,G\}$ of $K=|\mathcal R|$ candidate transcriptional regulators. For each experimental condition $q\in\{0,\ldots,Q\}$ and biological time $t\ge0$, let $u^{(q)}(t),s^{(q)}(t)\in\mathbb R_+^G$ denote the unspliced and spliced RNA abundances, respectively, and let $s_{\mathcal R}^{(q)}(t)\in\mathbb R_+^K$ denote the regulator state. Define the splicing and degradation matrices $B=\operatorname{diag}(\beta_1,\ldots,\beta_G)$ and $\Gamma=\operatorname{diag}(\gamma_1,\ldots,\gamma_G)$, where $\beta_g,\gamma_g>0$.

We consider the coupled transcription--splicing dynamics
\begin{align}
\dot u^{(q)}(t) &= a\bigl(s_{\mathcal R}^{(q)}(t),q\bigr)-Bu^{(q)}(t), \label{eq:uv-u}\\
\dot s^{(q)}(t) &= Bu^{(q)}(t)-\Gamma s^{(q)}(t), \label{eq:uv-s}
\end{align}
where $a=(a_1,\ldots,a_G)^\top$ denotes the transcription-rate function. Throughout the main theoretical development, we use the linear model
\begin{equation}
a_g\bigl(s_{\mathcal R}^{(q)},q\bigr)=c_g+h_g(q)+A_{g\cdot}M_qs_{\mathcal R}^{(q)}, \label{eq:linear-transcription}
\end{equation}
where $A\in\mathbb R^{G\times K}$ is the unknown regulatory matrix, $A_{g\cdot}$ denotes its $g$th row, $c_g$ is the basal transcription rate, $h_g(q)$ is a direct intervention effect treated as known in the theoretical analysis with $h_g(0)=0$, and $M_q:\mathbb R^K\to\mathbb R^K$ is a known perturbation operator specified by the intervention design and describes the effect of condition $q$ on regulator activity.

The entry $A_{gj}$ is positive for activation, negative for repression, and zero in the absence of a direct regulatory effect. We assume that $A$ is sparse. The known operator $M_q$ encodes the direct action of the intervention on regulator activity before it enters the transcriptional response. Thus, $s_{\mathcal R}^{(q)}(t)$ denotes the condition-specific regulator state, whereas $M_qs_{\mathcal R}^{(q)}(t)$ represents the effective regulator activity entering the transcriptional response. The framework accommodates, for example, transcription-factor knockout, knockdown, overexpression, and pharmacological inhibition. We take $M_0=I_K$, where $I_K$ is the $K\times K$ identity matrix, for the unperturbed control condition.

\subsection{Snapshot observation model}

For $N$ observed cells indexed by $i=1,\ldots,N$, let
\begin{equation}
Y_i=x^{(q_i)}(T_i)+\varepsilon_i,\qquad x^{(q)}(t)=\begin{pmatrix}u^{(q)}(t)\\ s^{(q)}(t)\end{pmatrix}, \label{eq:observation}
\end{equation}
where $Y_i\in\mathbb R^{2G}$ is the observed RNA snapshot, $\varepsilon_i\in\mathbb R^{2G}$ is the observation error, $q_i\in\{0,\ldots,Q\}$ is the experimental condition, $T_i\in[0,T_{\max}]$ is the latent biological time, with $T_{\max}>0$ denoting the maximal biological time considered, and $\mathbb E[\varepsilon_i\mid T_i,q_i]=0$. Because cells are measured destructively, each cell provides only a single snapshot, and the latent times must be estimated from the data.

\subsection{Forward and inverse problems}

Let $c=(c_1,\ldots,c_G)^\top$. With the kinetic parameters, intervention specification, and initial conditions fixed, the coupled dynamics define the parameter-to-state map
\begin{equation}
\mathcal F:(A,c)\longmapsto\{x^{(q)}(t)\}_{q=0}^Q.
\end{equation}
Composing $\mathcal F$ with the snapshot observation operator $\mathcal O$ gives the parameter-to-data map $\mathcal G=\mathcal O\circ\mathcal F$. Thus, schematically,
\begin{equation}
Y=\mathcal G(A,c)+\varepsilon,
\end{equation}
where $Y$ and $\varepsilon$ denote the collections of observations and observation errors, respectively.

Our objective is to recover the sparse regulatory matrix $A$ from the noisy snapshot observations $\{Y_i,q_i\}_{i=1}^N$. The problem combines snapshot observations of latent trajectories, latent biological time, and multiple perturbation conditions. The following sections characterize control-only nonidentifiability, determine when perturbations restore identifiability, and develop a sparse integral reconstruction method with finite-sample guarantees.

\section{Nonidentifiability of the control-only inverse problem}
\label{Nonidentifiability of the control-only inverse problem}

We first ask whether the regulatory matrix can be uniquely recovered from unperturbed RNA-velocity observations. Throughout this section, we restrict to the control condition $q=0$. The question is whether the regulatory parameters are uniquely determined by the transcriptional forcing along the observed control trajectory. If distinct regulatory parameters induce the same forcing along that trajectory, then they cannot be distinguished from control observations alone.

\subsection{Control regulator trajectory}

Let $s_{\mathcal R}^{(0)}(t)\in\mathbb R^K$ denote the regulator trajectory under the control condition, and define
\[
\mathcal M_0=\{s_{\mathcal R}^{(0)}(t):0\le t\le T_{\max}\}\subset\mathbb R^K.
\]
For a target gene $g$, the transcriptional forcing along this trajectory is
\[
f_g^{(0)}(t)=c_g+A_{g\cdot}s_{\mathcal R}^{(0)}(t).
\]
Thus, control observations determine the affine functional $s\mapsto c_g+A_{g\cdot}s$ only through its restriction to $\mathcal M_0$.

\subsection{Observational equivalence}

\begin{definition}[Observational equivalence]
Two parameter pairs $(c_g,A_{g\cdot})$ and $(\widetilde c_g,\widetilde A_{g\cdot})$ are observationally equivalent under the control condition if
\[
c_g+A_{g\cdot}s_{\mathcal R}^{(0)}(t)
=
\widetilde c_g+\widetilde A_{g\cdot}s_{\mathcal R}^{(0)}(t)
\]
for all $t\in[0,T_{\max}]$.
\end{definition}

Equivalently, with $\Delta c_g=\widetilde c_g-c_g$ and $\Delta A_{g\cdot}=\widetilde A_{g\cdot}-A_{g\cdot}$,
\[
\Delta c_g+\Delta A_{g\cdot}s_{\mathcal R}^{(0)}(t)=0
\qquad\text{for all }t\in[0,T_{\max}].
\]

\subsection{Constructive nonidentifiability}

Define the augmented feature vector
\[
\phi^{(0)}(t)=
\begin{pmatrix}
1\\
s_{\mathcal R}^{(0)}(t)
\end{pmatrix}
\in\mathbb R^{K+1},
\qquad
\mathcal S_0=\operatorname{span}\{\phi^{(0)}(t):0\le t\le T_{\max}\}.
\]

\begin{theorem}[Constructive control-only nonidentifiability]
If $\dim(\mathcal S_0)<K+1$, then for every target gene $g$ there exists a nonzero perturbation $(\Delta c_g,\Delta A_{g\cdot})$ such that
\[
\Delta c_g+\Delta A_{g\cdot}s_{\mathcal R}^{(0)}(t)=0
\qquad\text{for all }t\in[0,T_{\max}].
\]
Hence $(c_g,A_{g\cdot})$ and $(c_g+\Delta c_g,A_{g\cdot}+\Delta A_{g\cdot})$ are observationally equivalent under the control condition and cannot be distinguished from control observations alone.
\end{theorem}

\begin{proof}
Let
\[
\Delta\theta_g=
\begin{pmatrix}
\Delta c_g\\
\Delta A_{g\cdot}^{\top}
\end{pmatrix}
\in\mathbb R^{K+1}.
\]
Since $\dim(\mathcal S_0)<K+1$, the orthogonal complement $\mathcal S_0^\perp$ is nontrivial. Choose any nonzero $\Delta\theta_g\in\mathcal S_0^\perp$. Then
\[
\langle\Delta\theta_g,\phi^{(0)}(t)\rangle=0
\qquad\text{for all }t\in[0,T_{\max}],
\]
which is equivalent to
\[
\Delta c_g+\Delta A_{g\cdot}s_{\mathcal R}^{(0)}(t)=0
\qquad\text{for all }t\in[0,T_{\max}].
\]
Thus the perturbed and original parameters produce identical transcriptional forcing along the control trajectory.
\end{proof}

The theorem shows that control-only nonidentifiability is geometric rather than statistical. Any parameter direction in $\mathcal S_0^\perp$ is invisible along the control trajectory, so no estimator or regularization method can identify that component from control data alone without additional structural assumptions. Additional experimental information is therefore required.

\section{Intervention-assisted identifiability}
\label{Intervention-assisted identifiability}

We now characterize how controlled perturbations can restore identifiability. Measurements are assumed to be available under conditions $q=0,\ldots,Q$, each generating a regulator trajectory $s_{\mathcal R}^{(q)}(t)$. The key question is whether these trajectories collectively excite all regulatory parameter directions.

\subsection{Aggregated feature space}

For each condition $q$, define
\[
\phi^{(q)}(t)=
\begin{pmatrix}
1\\
M_qs_{\mathcal R}^{(q)}(t)
\end{pmatrix}
\in\mathbb R^{K+1},
\qquad
\mathcal S_q=\operatorname{span}\{\phi^{(q)}(t):0\le t\le T_{\max}\}.
\]
The aggregated feature space is
\[
\mathcal S_{\mathrm{all}}
=
\operatorname{span}\bigcup_{q=0}^Q\mathcal S_q
=
\operatorname{span}\{\phi^{(q)}(t):q=0,\ldots,Q,\ t\in[0,T_{\max}]\}.
\]
Thus, perturbations restore identifiability precisely when the combined trajectories span the full parameter space.

\subsection{Exact identifiability}

Define the population information matrix
\[
I
=
\sum_{q=0}^Q
\int_0^{T_{\max}}
\phi^{(q)}(t)\phi^{(q)}(t)^\top w_q(t)\,dt,
\qquad w_q(t)>0.
\]
For any $v\in\mathbb R^{K+1}$,
\[
v^\top Iv
=
\sum_{q=0}^Q
\int_0^{T_{\max}}
\bigl(v^\top\phi^{(q)}(t)\bigr)^2w_q(t)\,dt,
\]
and therefore
\[
I\succ0
\quad\Longleftrightarrow\quad
\mathcal S_{\mathrm{all}}=\mathbb R^{K+1}.
\]

\begin{theorem}[Intervention-assisted identifiability]
Suppose that the state trajectories $u^{(q)}(t)$ and $s^{(q)}(t)$ are known for $q=0,\ldots,Q$, and that the kinetic parameters, perturbation operators, and direct intervention effects are known. Then, for each target gene $g$, the linear regulatory parameters $(c_g,A_{g\cdot})$ are uniquely identifiable if and only if
\[
I\succ0.
\]
Equivalently, identifiability holds if and only if $\mathcal S_{\mathrm{all}}=\mathbb R^{K+1}$.
\end{theorem}

\begin{proof}
Suppose that $(c_g,A_{g\cdot})$ and $(\widetilde c_g,\widetilde A_{g\cdot})$ produce the same transcriptional forcing along every observed trajectory. With
\[
v=
\begin{pmatrix}
\widetilde c_g-c_g\\
(\widetilde A_{g\cdot}-A_{g\cdot})^\top
\end{pmatrix},
\]
observational equivalence is equivalent to
\[
v^\top\phi^{(q)}(t)=0
\qquad
\text{for all }q=0,\ldots,Q,\quad t\in[0,T_{\max}].
\]
Hence $Iv=0$.

If $I\succ0$, then $v=0$, so the two parameter pairs coincide. Conversely, if $I$ is singular, there exists $v\neq0$ with $Iv=0$. Since
\[
0=v^\top Iv
=
\sum_{q=0}^Q
\int_0^{T_{\max}}
\bigl(v^\top\phi^{(q)}(t)\bigr)^2w_q(t)\,dt
\]
and $w_q(t)>0$, we have $v^\top\phi^{(q)}(t)=0$ for all $q$ and $t$. The corresponding nonzero parameter perturbation is therefore observationally invisible under every perturbation condition.
\end{proof}

\subsection{Local nonlinear identifiability}

For a nonlinear transcription model
\[
a_g^{(q)}(t;\theta_g)
=
a_g\bigl(s_{\mathcal R}^{(q)}(t),q;\theta_g\bigr),
\]
define the sensitivity vector
\[
\psi_g^{(q)}(t)
=
\frac{\partial a_g^{(q)}(t;\theta_g)}{\partial\theta_g}
\]
and the sensitivity information matrix
\[
I_g^{\mathrm{sens}}
=
\sum_{q=0}^Q
\int_0^{T_{\max}}
\psi_g^{(q)}(t)\psi_g^{(q)}(t)^\top w_q(t)\,dt.
\]

\begin{proposition}[Local first-order identifiability]
Suppose that the transcription map is continuously differentiable with respect to $\theta_g$. If
\[
\lambda_{\min}\bigl(I_g^{\mathrm{sens}}\bigr)>0,
\]
then no nonzero infinitesimal parameter perturbation leaves the transcriptional forcing unchanged along all observed perturbation trajectories.
\end{proposition}

\begin{proof}
For any parameter direction $v$,
\[
v^\top I_g^{\mathrm{sens}}v
=
\sum_{q=0}^Q
\int_0^{T_{\max}}
\bigl(v^\top\psi_g^{(q)}(t)\bigr)^2w_q(t)\,dt.
\]
If $I_g^{\mathrm{sens}}\succ0$, this quantity is positive for every $v\neq0$. Hence no nonzero parameter direction lies in the null space of the linearized parameter-to-forcing map.
\end{proof}

\subsection{Practical identifiability and conditioning}
\label{Practical identifiability and conditioning}

Structural identifiability guarantees uniqueness but not numerical stability. Stability is governed by the spectrum of $I$. A small $\lambda_{\min}(I)$ indicates a weakly excited parameter direction and strong amplification of observational error, whereas a larger $\lambda_{\min}(I)$ corresponds to better conditioning. Equivalently, a large spectral condition number
\[
\operatorname{cond}_2(I)
=
\frac{\lambda_{\max}(I)}{\lambda_{\min}(I)}
\]
indicates near-nonidentifiability. These quantities therefore connect perturbation-induced excitation with the stability of subsequent parameter reconstruction.

\section{Integral sparse reconstruction}
\label{Integral sparse reconstruction}

We now turn from identifiability to reconstruction from noisy snapshot observations. Direct use of the differential equation requires estimating time derivatives of reconstructed trajectories, which can strongly amplify measurement and latent-time errors. We therefore use an equivalent integral formulation that avoids numerical differentiation.

\subsection{Integral reformulation}

For a target gene $g$ under perturbation condition $q$, the transcription dynamics satisfy
\begin{equation}
\dot u_g^{(q)}(t)
=
c_g+h_g(q)+A_{g\cdot}M_qs_{\mathcal R}^{(q)}(t)-\beta_gu_g^{(q)}(t).
\label{eq:target-transcription}
\end{equation}
Integrating \eqref{eq:target-transcription} over $[t_a,t_b]\subset[0,T_{\max}]$ gives
\begin{align}
u_g^{(q)}(t_b)-u_g^{(q)}(t_a)
+\beta_g\int_{t_a}^{t_b}u_g^{(q)}(t)\,dt
-\int_{t_a}^{t_b}h_g(q)\,dt
&=
c_g(t_b-t_a)
+
A_{g\cdot}\int_{t_a}^{t_b}M_qs_{\mathcal R}^{(q)}(t)\,dt.
\label{eq:integral-identity}
\end{align}
Thus the unknown regulatory parameters remain linear after integration. With
\[
\theta_g^\star=
\begin{pmatrix}
c_g\\
A_{g\cdot}^\top
\end{pmatrix},
\qquad
\Phi_q(t_a,t_b)=
\begin{pmatrix}
t_b-t_a\\
\displaystyle\int_{t_a}^{t_b}M_qs_{\mathcal R}^{(q)}(t)\,dt
\end{pmatrix},
\]
the integral identity can be written as
\[
\mathcal Y_g(t_a,t_b,q)
=
\Phi_q(t_a,t_b)^\top\theta_g^\star,
\]
where $\mathcal Y_g(t_a,t_b,q)$ denotes the left-hand side of \eqref{eq:integral-identity}. The dynamical reconstruction problem is thereby reduced to a linear inverse problem involving trajectory integrals rather than time derivatives.

\subsection{Empirical inverse problem}

In practice, the trajectories and kinetic parameters are estimated from snapshot data. Let $\widehat u_g^{(q)}(t)$, $\widehat s_{\mathcal R}^{(q)}(t)$, and $\widehat\beta_g$ denote these estimates. For an integration interval $[t_r,t_{r+1}]$, define
\begin{align}
\widehat y_{g,r}^{(q)}
&=
\widehat u_g^{(q)}(t_{r+1})
-\widehat u_g^{(q)}(t_r)
+\widehat\beta_g\int_{t_r}^{t_{r+1}}\widehat u_g^{(q)}(t)\,dt
-\int_{t_r}^{t_{r+1}}h_g(q)\,dt,
\label{eq:empirical-response}\\
\widehat x_r^{(q)}
&=
\begin{pmatrix}
t_{r+1}-t_r\\
\displaystyle\int_{t_r}^{t_{r+1}}M_q\widehat s_{\mathcal R}^{(q)}(t)\,dt
\end{pmatrix}
\in\mathbb R^{K+1}.
\label{eq:empirical-feature}
\end{align}
Stacking these quantities over perturbation conditions and integration intervals gives
\begin{equation}
\widehat y_g
=
\widehat X\theta_g^\star+\xi_g,
\label{eq:empirical-linear-model}
\end{equation}
where $\xi_g$ collects errors induced by trajectory estimation, latent-time uncertainty, kinetic-parameter estimation, numerical quadrature, measurement noise, and model misspecification. Thus both the response and the design matrix are constructed from estimated dynamical quantities.

\subsection{Sparse reconstruction}

Because the regulatory network is assumed sparse, we estimate $\theta_g^\star$ using
\begin{equation}
\widehat\theta_g
=
\arg\min_{\theta\in\mathbb R^{K+1}}
\left\{
\frac{1}{2M}\|\widehat y_g-\widehat X\theta\|_2^2
+
\lambda\|\theta_{-0}\|_1
\right\},
\label{eq:integral-lasso}
\end{equation}
where $M$ is the total number of stacked integral equations and $\theta_{-0}$ denotes the regulatory coefficients excluding the basal-transcription coefficient. The basal transcription rate is therefore unpenalized, while the $\ell_1$ penalty promotes sparsity in the regulatory interactions.

The reconstruction is conditional on upstream estimates of the latent trajectories and kinetic parameters; no particular trajectory-estimation procedure is assumed. Their estimation errors enter through the empirical design and effective error and are quantified in the recovery analysis below. The complete reconstruction procedure is summarized in Algorithm~\ref{alg:reconstruction}.

\begin{algorithm}[t]
\caption{Conditional integral sparse reconstruction}
\label{alg:reconstruction}
\begin{algorithmic}[1]
\Require Estimated trajectories $\{\widehat u^{(q)}(t),\widehat s_{\mathcal R}^{(q)}(t)\}_{q=0}^{Q}$, kinetic parameters $\{\widehat\beta_g\}_{g=1}^{G}$, perturbation operators $\{M_q\}_{q=0}^{Q}$, direct intervention effects $\{h_g(q)\}$, integration intervals $\{[t_r,t_{r+1}]\}_{r=1}^{R}$, and regularization parameter $\lambda$
\Ensure Estimated regulatory matrix $\widehat A$
\For{each target gene $g$}
    \For{each condition $q$ and interval $[t_r,t_{r+1}]$}
        \State Construct $\widehat y_{g,r}^{(q)}$ and $\widehat x_r^{(q)}$ using \eqref{eq:empirical-response}--\eqref{eq:empirical-feature}.
    \EndFor
    \State Stack the responses and features to form $\widehat y_g$ and $\widehat X$.
    \State Compute $\widehat\theta_g$ from \eqref{eq:integral-lasso}.
    \State Extract $\widehat c_g$ and $\widehat A_{g\cdot}$ from $\widehat\theta_g$.
\EndFor
\State Assemble $\widehat A$ from $\{\widehat A_{g\cdot}\}_{g=1}^{G}$.
\State \Return $\widehat A$
\end{algorithmic}
\end{algorithm}

\subsection{Effective error}

For subsequent analysis, we conceptually decompose the effective error as
\[
\xi_g
=
\xi_{g,\mathrm{count}}
+
\xi_{g,\mathrm{time}}
+
\xi_{g,\mathrm{kin}}
+
\xi_{g,\mathrm{quad}}
+
\xi_{g,\mathrm{model}},
\]
corresponding respectively to measurement noise, latent-time error, kinetic-parameter error, quadrature error, and model misspecification. The recovery theory requires control of their aggregate contribution through the score condition
\begin{equation}
\left\|
\frac{1}{M}\widehat X^\top\xi_g
\right\|_\infty
\le
\frac{\lambda}{2}.
\label{eq:effective-score}
\end{equation}
This decomposition is conceptual and need not be unique; the individual components need not be separately observable. It is used to distinguish stochastic errors that may decrease with increasing information from systematic errors that can persist.

\section{Finite-sample recovery theory}
\label{Finite-sample recovery theory}

We analyze reconstruction when the latent trajectories and kinetic parameters are estimated rather than known exactly. The resulting problem is a perturbed sparse inverse problem in which both the response and the design matrix are estimated. Our analysis is conditional on the upstream estimators and requires only bounds on the perturbations they induce.

\subsection{Perturbed inverse problem}

Let $X$ and $y_g$ denote the exact integral design matrix and response vector constructed from the true trajectories and kinetic parameter $\beta_g$. By the integral identity,
\[
y_g=X\theta_g^\star,
\qquad
\theta_g^\star=
\begin{pmatrix}
c_g^\star\\
A_{g\cdot}^{\star\top}
\end{pmatrix}.
\]
Write the empirical quantities as
\[
\widehat X=X+\Delta_X,
\qquad
\widehat y_g=y_g+\Delta_{y,g}.
\]
Then
\[
\widehat y_g-\widehat X\theta_g^\star
=
\Delta_{y,g}-\Delta_X\theta_g^\star.
\]
Defining
\begin{equation}
\xi_g:=\Delta_{y,g}-\Delta_X\theta_g^\star,
\label{eq:effective-error}
\end{equation}
we obtain the empirical model
\begin{equation}
\widehat y_g=\widehat X\theta_g^\star+\xi_g.
\label{eq:perturbed-inverse}
\end{equation}
Thus errors in the response enter through $\Delta_{y,g}$, whereas errors in the estimated design act through $-\Delta_X\theta_g^\star$.

The empirical Gram matrix is also stable under small design perturbations. Define
\[
\Sigma=\frac{1}{M}X^\top X,
\qquad
\widehat\Sigma=\frac{1}{M}\widehat X^\top\widehat X.
\]
If $\|\Delta_X\|_2\le\delta$, then
\begin{equation}
\lambda_{\min}(\widehat\Sigma)
\ge
\lambda_{\min}(\Sigma)
-
\frac{2\|X\|_2\delta+\delta^2}{M}.
\label{eq:gram-stability}
\end{equation}
The proof is given in Appendix~\ref{app:gram-stability}. In sparse high-dimensional settings, however, global positive definiteness is stronger than necessary; the recovery analysis below requires only restricted curvature.

\subsection{Sparse recovery guarantee}

For a target gene $g$, let $S_g=\operatorname{supp}(A_{g\cdot}^\star)$ and assume $|S_g|\le d$. For $v\in\mathbb R^{K+1}$, let $v_0$ denote its basal-transcription coefficient component and let $v_{S_g}$ and $v_{S_g^c}$ denote its regulatory components on $S_g$ and $S_g^c$, respectively. For example, if $A_{g\cdot}^\star=(0,2,0,-1,0)$, then $S_g=\operatorname{supp}(A_{g\cdot}^\star)=\{2,4\}$. For $v=(v_0,v_1,\ldots,v_5)$, this gives $v_{S_g}=(v_2,v_4)$ and $v_{S_g^c}=(v_1,v_3,v_5)$.

Assume that the empirical design satisfies the restricted strong convexity condition
\begin{equation}
\frac{1}{M}\|\widehat Xv\|_2^2
\ge
\kappa\|v\|_2^2
\label{eq:rsc}
\end{equation}
for every $v$ in the cone
\begin{equation}
\|v_{S_g^c}\|_1
\le
3\|v_{S_g}\|_1+|v_0|.
\label{eq:rsc-cone}
\end{equation}

\begin{theorem}[Sparse recovery bound]
\label{thm:finite_recovery}
Suppose that $A_{g\cdot}^\star$ is $d$-sparse, \eqref{eq:rsc} holds with $\kappa>0$, and
\begin{equation}
\left\|
\frac{1}{M}\widehat X^\top\xi_g
\right\|_\infty
\le
\frac{\lambda}{2}.
\label{eq:score-condition}
\end{equation}
Then the estimator defined in \eqref{eq:integral-lasso} satisfies
\begin{equation}
\|\widehat\theta_g-\theta_g^\star\|_2
\le
\frac{\sqrt{9d+1}}{\kappa}\lambda.
\label{eq:recovery-bound}
\end{equation}
\end{theorem}

The proof is given in Appendix~\ref{app:sparse-recovery-proof}.

\begin{corollary}[Recovery under stochastic and systematic error]
\label{cor:recovery-error}
Suppose, in addition, that with probability at least $1-\eta$,
\begin{equation}
\left\|
\frac{1}{M}\widehat X^\top\xi_g
\right\|_\infty
\le
C_0
\left(
\sigma\sqrt{\frac{\log K}{M}}
+
\delta_{\mathrm{time}}
+
\delta_{\mathrm{kin}}
+
\delta_{\mathrm{quad}}
+
\delta_{\mathrm{model}}
\right).
\label{eq:aggregate-score-bound}
\end{equation}
Then choosing
\[
\lambda
=
2C_0
\left(
\sigma\sqrt{\frac{\log K}{M}}
+
\delta_{\mathrm{time}}
+
\delta_{\mathrm{kin}}
+
\delta_{\mathrm{quad}}
+
\delta_{\mathrm{model}}
\right)
\]
gives, with probability at least $1-\eta$,
\begin{equation}
\|\widehat\theta_g-\theta_g^\star\|_2
\le
\frac{2C_0\sqrt{9d+1}}{\kappa}
\left(
\sigma\sqrt{\frac{\log K}{M}}
+
\delta_{\mathrm{time}}
+
\delta_{\mathrm{kin}}
+
\delta_{\mathrm{quad}}
+
\delta_{\mathrm{model}}
\right).
\label{eq:aggregate-recovery}
\end{equation}
\end{corollary}

\subsection{Interpretation}

The bound \eqref{eq:aggregate-recovery} separates a stochastic term, $\sigma\sqrt{\log K/M}$, from systematic contributions due to latent-time error, kinetic-parameter error, numerical quadrature, and model misspecification. Here $M$ is the number of stacked integral equations rather than the number of measured cells. The stochastic contribution decreases as $M$ increases, whereas the systematic terms need not vanish with additional data.

The constant $\kappa$ measures restricted conditioning of the empirical inverse problem. Perturbations that generate sufficiently diverse regulator trajectories strengthen curvature along sparse parameter directions, while redundant perturbations can leave nearly unobservable directions and amplify reconstruction error. Accurate recovery therefore requires both informative perturbations and sufficiently accurate upstream trajectory and kinetic-parameter estimation.

\section{Numerical experiments}
\label{Numerical experiments}

We illustrate the identifiability, conditioning, and reconstruction results developed above. Unless otherwise stated, the synthetic experiments use
\begin{align}
\dot u_g^{(q)}(t)
&=
c_g+h_g(q)+A_{g\cdot}^{\star}M_qs_{\mathcal R}^{(q)}(t)
-\beta_gu_g^{(q)}(t), \\
\dot s_g^{(q)}(t)
&=
\beta_gu_g^{(q)}(t)-\gamma_gs_g^{(q)}(t).
\label{eq:num_system}
\end{align}
Here $A^\star$ denotes the ground-truth regulatory matrix used to generate the synthetic data. For the identifiability and stability experiments, we set $G=4$, $K=3$, $\mathcal R=\{1,2,3\}$, and take $g=4$ as the target gene. The kinetic and basal-transcription parameters are
\[
\beta=(1,1,1,1.2)^\top,\qquad
\gamma=(0.4,0.4,0.4,0.6)^\top,\qquad
c=(1,1,1,0.8)^\top.
\]
with
\[
A_{1\cdot}^{\star}
=
A_{2\cdot}^{\star}
=
A_{3\cdot}^{\star}
=
(0.15,0.10,0.05),
\qquad
A_{4\cdot}^{\star}
=
(0.8,-0.3,0.5).
\]
We use the common initial conditions
\[
u^{(q)}(0)=(0.10,0.10,0.10,0.20)^\top,
\qquad
s^{(q)}(0)=(0.05,0.05,0.05,0.10)^\top
\]
for all conditions $q$, and simulate on $[0,10]$.

We set $h_g(q)=0$ for all genes and conditions, so that the synthetic perturbations act only through the operators $M_q$. The control condition is $M_0=I_3$. Regulator-specific knockdowns are represented by
\begin{equation}
M_1=\operatorname{diag}(0.2,1,1),
\qquad
M_2=\operatorname{diag}(1,0.2,1),
\qquad
M_3=\operatorname{diag}(1,1,0.2).
\label{eq:num_common_perturbations}
\end{equation}

\subsection{Near-nonidentifiability under control observations}

We first examine how control-only conditioning deteriorates as the regulator trajectories approach rank deficiency. Starting from synchronized regulator dynamics, we perturb the second and third regulatory rows according to
\[
A_{2\cdot}^{(\varepsilon)}
=
A_{1\cdot}^{\star}+\varepsilon d_2,
\qquad
A_{3\cdot}^{(\varepsilon)}
=
A_{1\cdot}^{\star}+\varepsilon d_3,
\]
where
\[
d_2=(0.08,-0.03,0.02),
\qquad
d_3=(-0.05,0.06,-0.01).
\]
The first and target-gene regulatory rows remain fixed at
$A_{1\cdot}^{(\varepsilon)}=A_{1\cdot}^{\star}$ and
$A_{4\cdot}^{(\varepsilon)}=A_{4\cdot}^{\star}$.
We also perturb the common initial conditions according to
\[
u^{(\varepsilon)}(0)
=
u(0)+\varepsilon(0,0.03,-0.02,0)^\top,
\qquad
s^{(\varepsilon)}(0)
=
s(0)+\varepsilon(0,-0.02,0.025,0)^\top.
\]
Thus $\varepsilon\to0$ approaches the rank-deficient control system.

\begin{figure}[t]
    \centering
    \includegraphics[width=\textwidth]{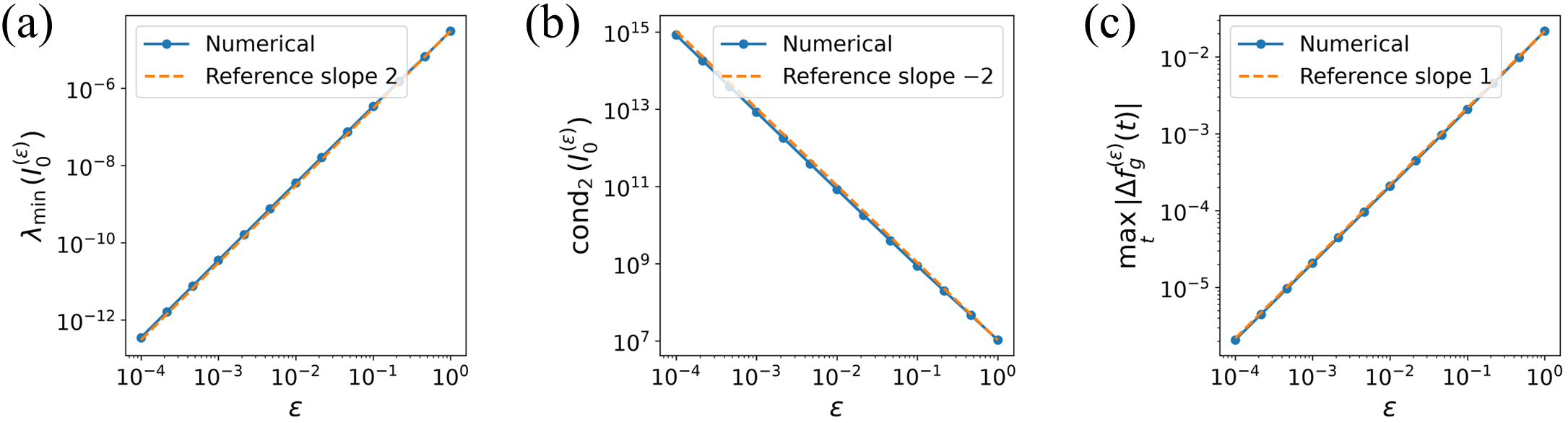}
    \vspace{-2em}
    \caption{Near-nonidentifiability under control observations.
    (a) Smallest eigenvalue of the information matrix.
    (b) Spectral condition number.
    (c) Maximum forcing discrepancy under a fixed parameter perturbation.}
    \label{fig:near_nonidentifiability}
\end{figure}

For $\varepsilon\in[10^{-4},1]$, define the augmented control feature vector
\[
\phi^{(\varepsilon)}(t)
=
\begin{pmatrix}
1\\
s_{\mathcal R}^{(\varepsilon)}(t)
\end{pmatrix}
\]
and compute
\[
I_0^{(\varepsilon)}
=
\int_0^{T_{\max}}
\phi^{(\varepsilon)}(t)
\phi^{(\varepsilon)}(t)^\top\,dt.
\]
As shown in Fig.~\ref{fig:near_nonidentifiability}(a), the smallest eigenvalue satisfies numerically
\[
\lambda_{\min}\bigl(I_0^{(\varepsilon)}\bigr)
\propto
\varepsilon^{1.985}.
\]
Correspondingly, Fig.~\ref{fig:near_nonidentifiability}(b) shows that
\[
\operatorname{cond}_2\bigl(I_0^{(\varepsilon)}\bigr)
\propto
\varepsilon^{-1.979},
\]
demonstrating rapid deterioration of conditioning near the nonidentifiable limit.

To examine the corresponding weak parameter direction, let
$v_{\min}^{(\varepsilon)}$ be a unit eigenvector associated with
$\lambda_{\min}(I_0^{(\varepsilon)})$ and define
\[
\Delta\theta_g^{(\varepsilon)}
=
\alpha v_{\min}^{(\varepsilon)},
\qquad
\widetilde\theta_g^{(\varepsilon)}
=
\theta_g^\star+\Delta\theta_g^{(\varepsilon)},
\]
where $\alpha>0$ controls the magnitude of the parameter perturbation. In the numerical experiment, we set $\alpha=2$. The resulting forcing discrepancy is
\[
\Delta f_g^{(\varepsilon)}(t)
=
\bigl(\Delta\theta_g^{(\varepsilon)}\bigr)^\top
\phi^{(\varepsilon)}(t).
\]
Although the parameter perturbation has fixed norm
$\|\Delta\theta_g^{(\varepsilon)}\|_2=\alpha$, the maximum forcing discrepancy
$\max_{t\in[0,T_{\max}]}|\Delta f_g^{(\varepsilon)}(t)|$
scales approximately as $\varepsilon^{1.003}$, as shown in
Fig.~\ref{fig:near_nonidentifiability}(c). Moreover,
\[
\|\Delta f_g^{(\varepsilon)}\|_{L^2(0,T_{\max})}^2
=
\alpha^2
\lambda_{\min}\bigl(I_0^{(\varepsilon)}\bigr),
\]
which explains why a fixed parameter perturbation becomes increasingly difficult to distinguish as the system approaches the nonidentifiable limit.

\subsection{Perturbation-assisted identifiability}

We next consider the exactly degenerate system with
$s_1^{(0)}(t)=s_2^{(0)}(t)=s_3^{(0)}(t)$.
Although the regulator trajectories remain synchronized within each condition, the effective regulator states
$M_qs_{\mathcal R}^{(q)}(t)$ generated by
\eqref{eq:num_common_perturbations} explore distinct directions, as shown in
Fig.~\ref{fig:perturbation_identifiability}(a).

For each condition $q$, define the augmented feature vector
\[
\phi^{(q)}(t)
=
\begin{pmatrix}
1\\
M_qs_{\mathcal R}^{(q)}(t)
\end{pmatrix}.
\]
We then define the sequentially aggregated information matrix
\begin{equation}
I_{\mathrm{all}}^{(m)}
=
\sum_{q=0}^{m}
\int_0^{T_{\max}}
\phi^{(q)}(t)\phi^{(q)}(t)^\top\,dt,
\qquad
m=0,1,2,3,
\label{eq:num_sequential_information}
\end{equation}
where $m$ denotes the number of perturbation conditions added to the control, so that $m=0$ corresponds to control only and $m=3$ includes all three perturbations.

The spectra of the sequentially aggregated information matrices are shown in
Fig.~\ref{fig:perturbation_identifiability}(b). Their corresponding ranks for
$m=0,1,2,3$ are
\[
2,\quad 3,\quad 4,\quad 4.
\]
As summarized in Fig.~\ref{fig:perturbation_identifiability}(c), the rank reaches
$K+1=4$ after the second perturbation is added, restoring identifiability. Adding the third perturbation does not further increase the rank, but increases
$\lambda_{\min}\bigl(I_{\mathrm{all}}^{(m)}\bigr)$ from approximately $4.64$ to $10.14$ and decreases
$\operatorname{cond}_2\bigl(I_{\mathrm{all}}^{(m)}\bigr)$ from approximately
$1.13\times10^2$ to $6.51\times10^1$.
Thus additional perturbations can improve conditioning even after structural identifiability has been achieved.

\begin{figure}[t]
    \centering
    \includegraphics[width=\textwidth]{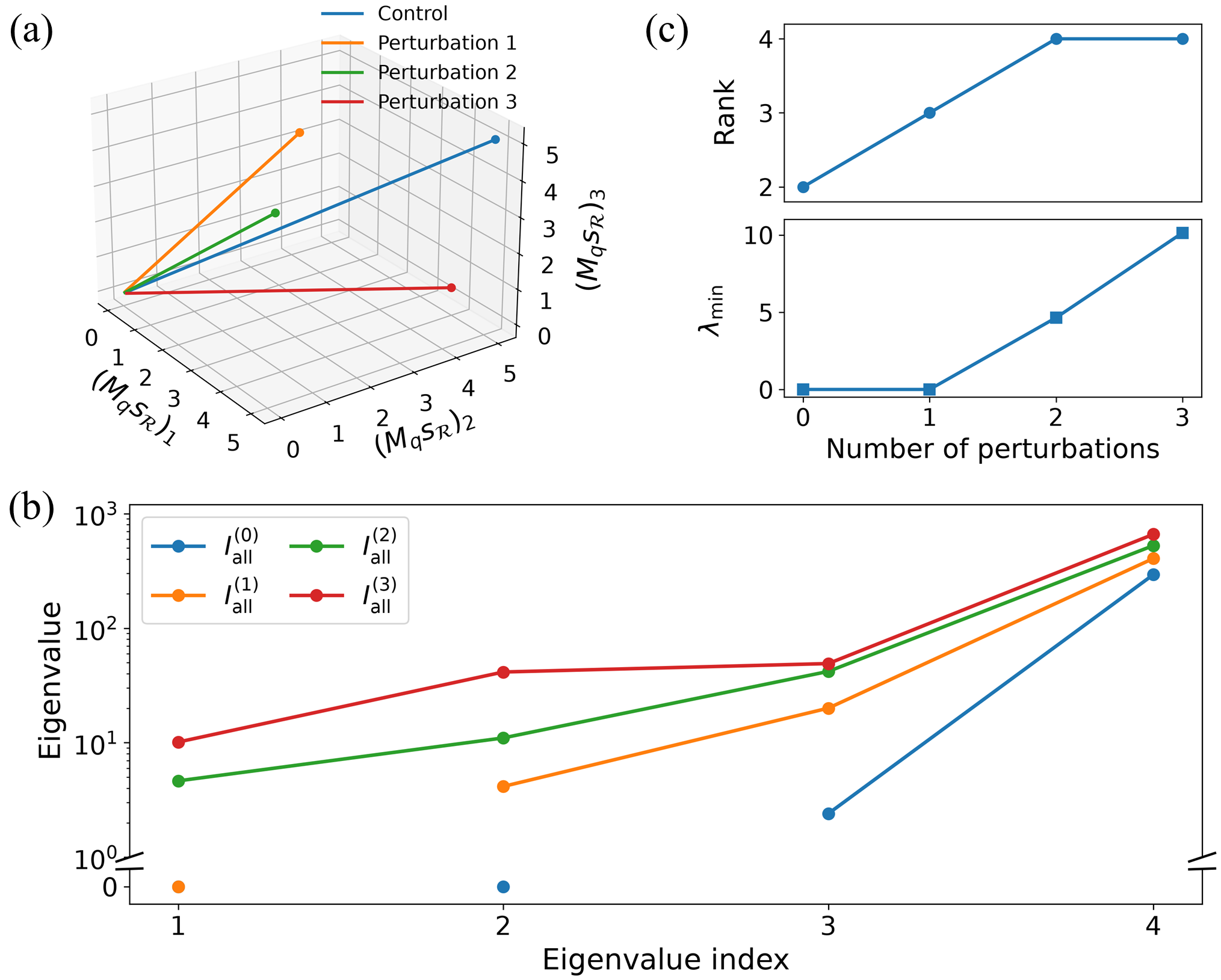}
    \vspace{-2em}
    \caption{Perturbation-assisted identifiability.
    (a) Effective regulator trajectories.
    (b) Spectra of the sequentially aggregated information matrices.
    (c) Rank and smallest eigenvalue as perturbations are added sequentially.}
    \label{fig:perturbation_identifiability}
\end{figure}

\subsection{Stability of integral reconstruction}

We compare differential and integral reconstruction under observation noise using all perturbation conditions above. Independent Gaussian noise is added separately to each unspliced and spliced trajectory component, with standard deviation equal to $\sigma$ times the temporal standard deviation of that component, where
\[
\sigma\in\{0,0.01,0.025,0.05,0.10,0.20\}.
\]
Both methods use the same noisy trajectories, which are smoothed using a Savitzky--Golay filter. For the differential formulation, derivatives are estimated by Savitzky--Golay local polynomial differentiation. The baseline integration-window width for the integral formulation is $\Delta=0.5$. Results are averaged over $50$ independent Monte Carlo realizations.

We first compare the relative equation residuals
\[
R_D=\frac{\|\widehat r_D\|_2}{\|y_D^\star\|_2},
\qquad
R_I=\frac{\|\widehat r_I\|_2}{\|y_I^\star\|_2},
\]
where $\widehat r_D$ and $\widehat r_I$ are the residual vectors obtained by evaluating the differential and integral equations, respectively, on the noisy smoothed trajectories using the ground-truth parameters, and $y_D^\star$ and $y_I^\star$ are the corresponding noise-free reference response vectors. As shown in Fig.~\ref{fig:integral_stability}(a), both residuals increase approximately linearly with the noise level, but the differential residual is roughly $4.6$--$4.7$ times larger over the nonzero noise levels considered. A representative realization at $\sigma=0.1$ is shown in Fig.~\ref{fig:integral_stability}(b), where numerical differentiation substantially amplifies observation noise and produces large fluctuations around the true derivative $\dot u_4(t)$.

We next compare parameter recovery using matched numbers of differential and integral equations. To isolate the effect of the two equation formulations, parameters are estimated by ordinary least squares in both cases. We measure the relative parameter error by
\[
E_\theta
=
\frac{\|\widehat\theta_g-\theta_g^\star\|_2}
{\|\theta_g^\star\|_2}.
\]
As shown in Fig.~\ref{fig:integral_stability}(c), the integral formulation has lower mean parameter error at every nonzero noise level, with the ratio $E_{\theta,D}/E_{\theta,I}$ ranging from $4.06$ to $4.90$, where $E_{\theta,D}$ and $E_{\theta,I}$ denote the relative parameter errors for the differential and integral formulations, respectively. Finally, Fig.~\ref{fig:integral_stability}(d) shows that the integral reconstruction is insensitive to the integration-window width over
\[
\Delta\in\{0.25,0.5,1.0\},
\]
with similar mean parameter errors across the three choices.

\begin{figure}[t]
    \centering
    \includegraphics[width=\textwidth]{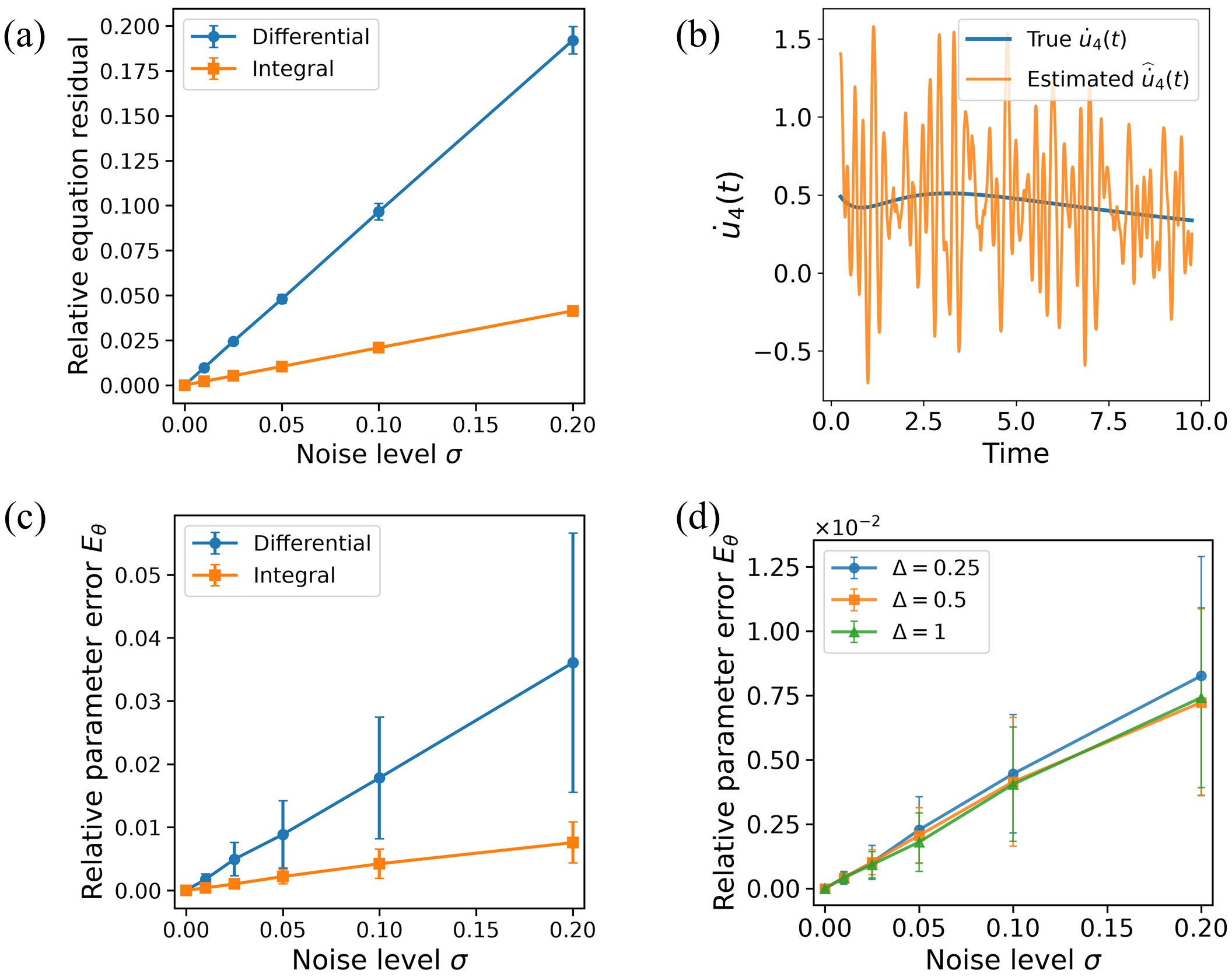}
    \vspace{-2em}
    \caption{Stability of integral reconstruction under observation noise.
    (a) Relative equation residuals for the differential and integral formulations.
    (b) Representative true and numerically estimated target-gene derivative at $\sigma=0.1$.
    (c) Relative parameter errors for the differential and integral formulations with matched equation counts.
    (d) Sensitivity of the integral reconstruction to the integration-window width.}
    \label{fig:integral_stability}
\end{figure}

\subsection{Finite-sample and systematic errors}

We next examine the two error regimes suggested by the recovery theory. First, to isolate stochastic error, we use $K=8$ candidate regulators, $d=3$ active regulatory coefficients, and one regulator-specific knockdown per candidate regulator. Exact trajectories generate a common pool of integral equations with integration-window width $\Delta=0.5$. For
\[
M\in\{50,100,200,400,800\},
\]
we sample $M$ equations without replacement from this pool and add independent Gaussian noise to the response. The noise standard deviation is
\[
\sigma_y=\sigma\,\operatorname{sd}(y^\star),
\qquad
\sigma=0.05,
\]
where $y^\star$ denotes the noise-free response over the full equation pool. The estimator uses
\[
\lambda_M
=
2\sigma_y\sqrt{\frac{\log K}{M}},
\]
with the basal-transcription coefficient unpenalized.

Over $100$ Monte Carlo realizations, the relative parameter error $E_\theta$, defined above, decreases with $M$, as shown in Fig.~\ref{fig:finite_sample_systematic_error}(a). We fit the empirical errors to the power law
\[
E_\theta(M)\approx C_{\mathrm{fit}}M^{-p},
\]
where $C_{\mathrm{fit}}>0$ is a fitted prefactor and $p>0$ is the fitted decay exponent. The resulting estimate is $\widehat p=0.603$. This empirical scaling is compatible with the $M^{-1/2}$ reference scale in the recovery bound, which is an upper-bound rate rather than an asserted asymptotic law.

\begin{figure}[t]
    \centering
    \includegraphics[width=\textwidth]{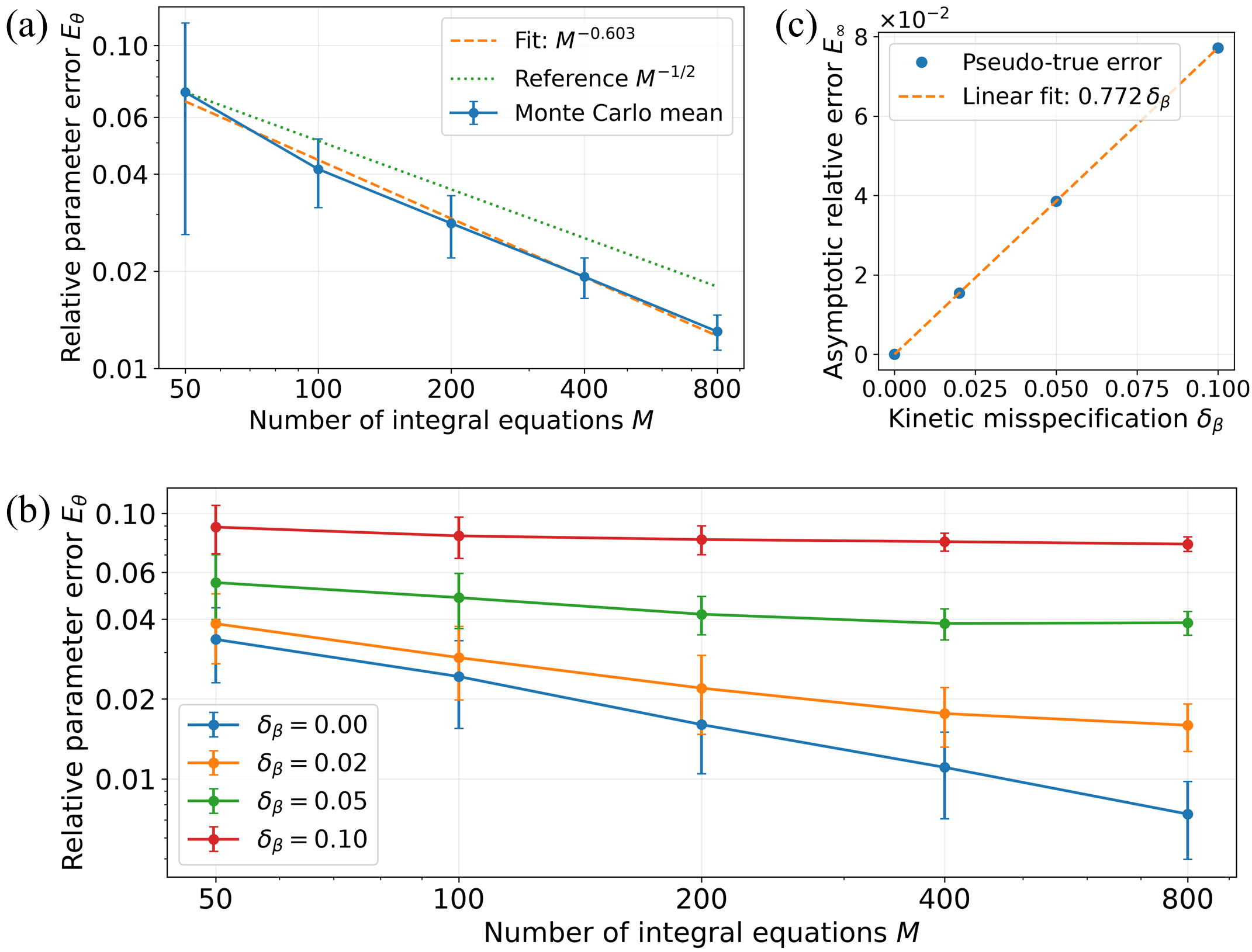}
    \vspace{-2em}
    \caption{Finite-sample and systematic errors in integral reconstruction.
    (a) Finite-sample relative parameter error and power-law scaling with the number $M$ of integral equations.
    (b) Relative parameter error versus $M$ under kinetic misspecification.
    (c) Pseudo-true asymptotic relative error versus the misspecification level $\delta_\beta$.}
    \label{fig:finite_sample_systematic_error}
\end{figure}

We next isolate systematic kinetic error in a separate synthetic inverse problem. Reconstruction is performed with a misspecified degradation rate
\begin{equation}
\widehat\beta_g
=
(1+\delta_\beta)\beta_g^\star,
\qquad
\delta_\beta\in\{0,0.02,0.05,0.10\},
\label{eq:num_beta_misspecification}
\end{equation}
where $\delta_\beta$ denotes the relative kinetic misspecification. For each value of $\delta_\beta$, we vary the number $M$ of sampled equations over the same set of values as above. As shown in Fig.~\ref{fig:finite_sample_systematic_error}(b), increasing $M$ reduces the finite-sample contribution to the reconstruction error. When $\delta_\beta=0$, the error continues to decrease with $M$, whereas for $\delta_\beta>0$ the curves approach nonzero levels determined by the kinetic misspecification.

To characterize these limiting errors, let $X_{\mathrm{pop}}$ denote the full synthetic design matrix used in the systematic-error experiment and let $y_{\delta_\beta}$ denote the corresponding response vector constructed using the misspecified rate $(1+\delta_\beta)\beta_g^\star$. We define the pseudo-true parameter by
\[
\theta_g^\dagger(\delta_\beta)
=
\arg\min_\theta
\|y_{\delta_\beta}-X_{\mathrm{pop}}\theta\|_2^2.
\]
The corresponding asymptotic relative error is
\[
E_\infty(\delta_\beta)
=
\frac{\|\theta_g^\dagger(\delta_\beta)-\theta_g^\star\|_2}
{\|\theta_g^\star\|_2}.
\]
As shown in Fig.~\ref{fig:finite_sample_systematic_error}(c), the pseudo-true error depends approximately linearly on the misspecification level over the tested range,
\[
E_\infty(\delta_\beta)
\approx
0.772\,\delta_\beta.
\]
Thus increasing the number of integral equations suppresses finite-sample reconstruction error but cannot eliminate the bias induced by systematic kinetic misspecification.

\subsection{Perturbation-resolved reconstruction in RPE1 cells}
\label{sec:rpe1}

We next apply the full integral reconstruction framework to perturbation-resolved single-cell RNA-velocity data from the RPE1 CRISPRi Perturb-seq experiment of Replogle et al. \cite{replogle2022mapping}, using VeloCycle-derived cell-cycle phases and kinetic parameters \cite{lederer2024statistical}. The availability of separate spliced and unspliced measurements permits application of the integral construction in Algorithm~5.1. After quality control, the analysis contains 151 perturbation conditions together with a non-targeting control, 151 candidate regulators, and 426 response genes. Cells are binned according to the inferred VeloCycle cell-cycle phase, yielding 1214 usable adjacent-phase integral equations across the 152 conditions. The reconstruction is conditional on the estimated latent trajectories and kinetic parameters. Details of data processing, trajectory construction, perturbation encoding, and numerical implementation are provided in the supplementary materials.

We first examine whether the perturbation panel supplies the excitation required by the identifiability analysis of Section~4. For the complete empirical integral design, including the basal-transcription column, the design matrix has full rank \(152\). The scaled information matrix has smallest eigenvalue \(\lambda_{\min}=2.67\times10^{-3}\) and condition number \(5.23\times10^{4}\). Thus the complete perturbation panel is structurally full rank, although the inverse problem remains nontrivially conditioned.

To examine how identifiability develops as perturbations are added, we generate 100 random perturbation orderings, retaining the non-targeting control throughout, and recompute the empirical rank and smallest information eigenvalue along each ordering. Full rank is first attained after a median of 42 perturbations, with mean \(41.86\) and range \(40\)--\(44\). Beyond the full-rank transition, additional perturbations continue to improve conditioning. The median smallest eigenvalue is approximately \(7\times10^{-6}\) at 50 perturbations and increases to \(2.67\times10^{-3}\) with all 151 perturbations; see Fig.~\ref{fig:rpe1-summary}(a)--(b). These observations are consistent with the theoretical distinction between structural identifiability and practical conditioning. Increasing perturbational diversity can first eliminate rank deficiency and subsequently improve conditioning even after full rank has been reached.

We next apply the integral sparse estimator~(5.6) to the complete design, with the basal-transcription coefficient left unpenalized. Twelve genes occur as both response genes and directly perturbed regulators; their own perturbation equations and self-regulatory predictors require special treatment because the direct intervention contribution is not independently known and same-source confounding would otherwise arise. The corresponding exclusions are described in the supplementary materials.

The resulting regulatory matrix contains 15,247 nonzero coefficients among 64,326 candidate regulator--response pairs, corresponding to a density of \(23.7\%\). To assess whether the selected structure is robust to the particular perturbation panel, we repeat the reconstruction over 100 random subsamples containing 80\% of the perturbation conditions, with the non-targeting control retained throughout. Among edges selected in the full-data reconstruction, 8,616 are recovered in at least 80\% of the subsampled fits, corresponding to 56.5\% of the full-data selected network. Moreover, 6,381 edges are recovered in at least 90\% of the fits, 5,157 in at least 95\%, and 2,999 in every subsample. For edges with selection frequency at least \(0.8\), the median sign agreement across fits in which the edge is selected is \(0.97\). As shown in Fig.~\ref{fig:rpe1-summary}(c), the reconstruction therefore contains a substantial network core that is robust to moderate changes in the available perturbation conditions. This stability measures reproducibility of the inverse reconstruction rather than biological correctness of individual selected edges.

\begin{figure}[t]
    \centering
    \includegraphics[width=\linewidth]{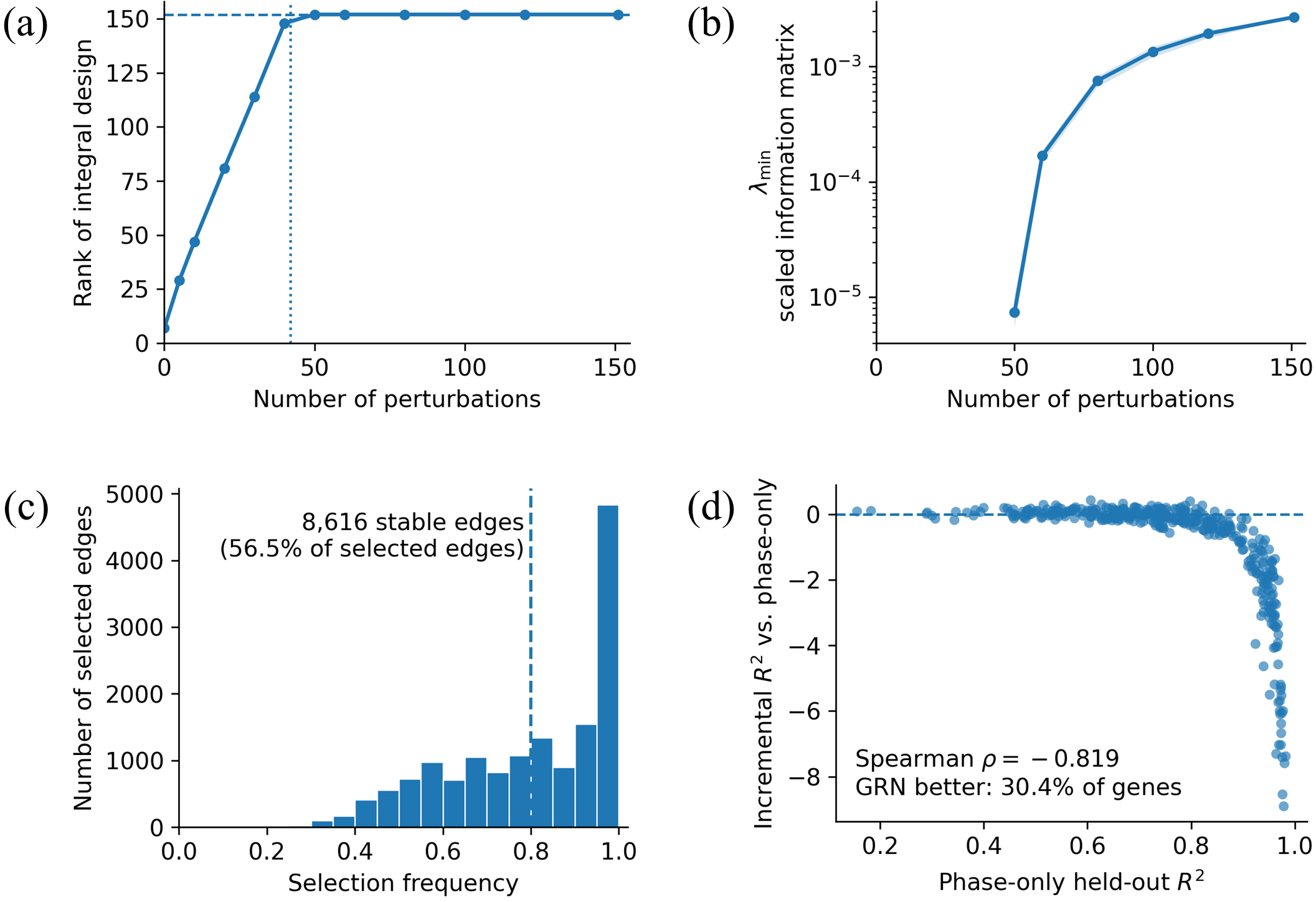}
    \vspace{-2em}
    \caption{Perturbation-resolved integral GRN reconstruction in RPE1 cells. (a) Empirical design rank as perturbation conditions are added across randomized perturbation orderings; reference lines indicate full rank and the median first-full-rank transition. (b) Smallest eigenvalue of the scaled information matrix along the same orderings. (c) Selection-frequency distribution of full-data selected edges under perturbation-condition subsampling; the dashed line marks the stability threshold of \(0.8\). (d) Incremental held-out predictive performance of the reconstructed GRN relative to the phase-only baseline as a function of phase-only predictability; the dashed line denotes zero incremental gain.}
    \label{fig:rpe1-summary}
\end{figure}

Finally, reconstruction stability does not imply predictive generalization to perturbation conditions excluded from estimation. We therefore perform a separate five-fold condition-held-out evaluation. In each fold, the regulatory matrix is estimated from the training perturbation conditions and used to predict the integral responses in the held-out conditions, conditional on their estimated regulator trajectories. The non-targeting condition is retained in the training set throughout. The 31 response genes used for regularization calibration are excluded from the primary evaluation, leaving 395 genes. The complete validation protocol is described in the supplementary materials.

Because RPE1 cells exhibit strong cell-cycle structure, we compare the GRN-based prediction with a phase-only baseline that captures response variation associated with inferred cell-cycle phase without using the reconstructed regulatory matrix. Across the 395 evaluation genes, the mean conventional held-out coefficients of determination are \(R^2_{\mathrm{GRN}}=0.716\) and \(R^2_{\mathrm{phase}}=0.757\), and the GRN model outperforms the phase-only baseline for 30.4\% of the genes. To quantify predictive improvement relative to the phase-only baseline, we define
\[
\Delta R^2_{\mathrm{phase}}
=
1-
\frac{\operatorname{SSE}_{\mathrm{GRN}}}
     {\operatorname{SSE}_{\mathrm{phase}}},
\]
so that \(\Delta R^2_{\mathrm{phase}}>0\) indicates improvement by the GRN model. The incremental score is strongly negatively associated with phase-only predictability, with Spearman correlation \(\rho=-0.819\). In particular, among the 58 genes for which \(R^2_{\mathrm{phase}}\geq0.95\), the GRN reconstruction does not outperform the phase-only baseline for any gene. Figure~\ref{fig:rpe1-summary}(d) therefore shows that regulatory information provides additional held-out predictive value for a subset of responses, but the present reconstruction does not generally improve upon the shared cell-cycle trajectory for strongly phase-dominated genes.

Taken together, the RPE1 experiment distinguishes three empirical properties of the inverse problem and its reconstruction. Increasing perturbational diversity leads to an empirically full-rank design and continues to improve conditioning beyond the full-rank transition; the resulting sparse integral reconstruction contains a substantial condition-subsampling-stable network core; and held-out evaluation shows that these properties do not imply uniform predictive gains over a strong cell-cycle baseline. The real-data results are therefore consistent with the proposed perturbation-assisted identifiability mechanism while also identifying strongly phase-dominated responses as an important limitation of conditional out-of-condition generalization.

\section{Discussion}
\label{sec:discussion}

We have formulated gene regulatory network reconstruction from RNA velocity as a sparse dynamical inverse problem. The main conclusion is that dynamical information alone does not ensure recoverability. Identifiability depends on whether the regulator trajectories sufficiently excite the unknown parameter directions. Control-only nonidentifiability is therefore structural rather than algorithmic, and cannot in general be resolved by changing the estimator or regularization scheme. Controlled perturbations can remove invisible directions by enlarging the aggregated feature space, consistent with the broader role of perturbations in network identifiability and experimental design \cite{zak2001simulation,gross2020identifiability}. Our analysis further distinguishes rank from conditioning. Perturbations may first restore structural identifiability and subsequently improve the smallest information eigenvalue without changing rank. This perspective is closely related to optimal experimental design, where informative experiments are selected to improve parameter identifiability or information criteria \cite{chaloner1995bayesian,pukelsheim2006optimal}. The synthetic and RPE1 experiments exhibit this separation, suggesting that perturbation design should account for both identifiability and conditioning.

For reconstruction, the integral formulation avoids numerical differentiation while preserving linear dependence on the regulatory parameters. This approach is related to integral and weak-form methods for data-driven dynamical-system identification, which improve robustness to noisy observations by replacing pointwise derivative estimation with integrated constraints \cite{schaeffer2017sparse,messenger2021weak}. It also belongs to the broader class of sparse system-identification methods in which parsimonious governing equations are recovered from data \cite{brunton2016discovering}. The synthetic experiments show substantially greater stability to observation noise than derivative-based reconstruction. The recovery theory complements this numerical advantage by separating finite-sample stochastic error from systematic contributions due to latent-time uncertainty, kinetic-parameter estimation, numerical quadrature, and model misspecification. Consequently, increasing the number of informative integral equations can reduce stochastic error but need not eliminate errors introduced upstream or by an imperfect dynamical model.

The RPE1 analysis illustrates both the applicability and the limitations of the framework. Increasing perturbational diversity produces an empirically full-rank design and improves conditioning, while perturbation-condition subsampling reveals a reproducible network core. These properties should not, however, be interpreted as validation of individual regulatory edges. Likewise, held-out prediction shows that identifiability and reconstruction stability do not imply uniform predictive improvement over a strong baseline. This distinction is relevant to RNA-velocity-based network inference, where dynamical information has previously been combined with sparse linear models to reconstruct regulatory interactions \cite{aubin2020gene}, as well as to the broader use of RNA velocity for estimating transcriptional dynamics from single-cell snapshots \cite{la2018rna,bergen2020generalizing}. In the present framework, identifiability guarantees uniqueness only within the assumed model and excitation regime; it does not guarantee model correctness or predictive optimality.

Several extensions are natural. Nonlinear transcription models would allow richer regulatory interactions but require corresponding extensions of the identifiability and recovery analysis. More directly, the information-matrix characterization suggests adaptive perturbation design in which interventions are selected to improve poorly constrained parameter directions, connecting the present framework to experimental-design approaches for dynamical and perturbational systems \cite{pukelsheim2006optimal, gross2020identifiability}. Improved latent-time, trajectory, and kinetic estimation would similarly reduce the effective errors entering the inverse problem. Overall, stable GRN reconstruction depends not simply on the number of observed cells, but on the information supplied by reconstructed dynamics, sparsity, and perturbational excitation.

\appendix
\section{Additional proofs}
\subsection{Proof of Eq.~\eqref{eq:gram-stability}}
\label{app:gram-stability}

Recall that $\widehat X=X+\Delta_X$ with $\|\Delta_X\|_2\le\delta$. Expanding the empirical Gram matrix gives
\begin{align*}
\widehat\Sigma-\Sigma
&=\frac{1}{M}\left[(X+\Delta_X)^\top(X+\Delta_X)-X^\top X\right]\\
&=\frac{1}{M}\left(X^\top\Delta_X+\Delta_X^\top X+\Delta_X^\top\Delta_X\right).
\end{align*}
Hence, by the triangle inequality and submultiplicativity of the spectral norm,
\begin{align*}
\|\widehat\Sigma-\Sigma\|_2
&\le \frac{1}{M}\left(\|X^\top\Delta_X\|_2+\|\Delta_X^\top X\|_2+\|\Delta_X^\top\Delta_X\|_2\right)\\
&\le \frac{1}{M}\left(2\|X\|_2\|\Delta_X\|_2+\|\Delta_X\|_2^2\right)\\
&\le \frac{2\|X\|_2\delta+\delta^2}{M}.
\end{align*}

To control the smallest eigenvalue, let $E=\widehat\Sigma-\Sigma$. Since $\Sigma$ and $\widehat\Sigma$ are symmetric, the Rayleigh--Ritz characterization gives
\[
\lambda_{\min}(\widehat\Sigma)=\min_{\|v\|_2=1}v^\top\widehat\Sigma v.
\]
For any $\|v\|_2=1$,
\[
v^\top\widehat\Sigma v
=v^\top\Sigma v+v^\top Ev
\ge \lambda_{\min}(\Sigma)-\|E\|_2,
\]
where $v^\top\Sigma v\ge\lambda_{\min}(\Sigma)$ and
\[
|v^\top Ev|\le \|v\|_2\|Ev\|_2\le\|E\|_2.
\]
Taking the minimum over $\|v\|_2=1$ therefore yields
\[
\lambda_{\min}(\widehat\Sigma)
\ge \lambda_{\min}(\Sigma)-\|\widehat\Sigma-\Sigma\|_2
\ge \lambda_{\min}(\Sigma)-\frac{2\|X\|_2\delta+\delta^2}{M},
\]
which proves \eqref{eq:gram-stability}.

\subsection{Detailed proof of Theorem~\ref{thm:finite_recovery}}
\label{app:sparse-recovery-proof}

\begin{proof}
Let $\Delta=\widehat\theta_g-\theta_g^\star$. By optimality of $\widehat\theta_g$ in \eqref{eq:integral-lasso},
\[
\frac{1}{2M}\|\widehat y_g-\widehat X\widehat\theta_g\|_2^2+\lambda\|\widehat\theta_{g,-0}\|_1
\le
\frac{1}{2M}\|\widehat y_g-\widehat X\theta_g^\star\|_2^2+\lambda\|\theta_{g,-0}^\star\|_1.
\]
Using \eqref{eq:perturbed-inverse}, we have
\[
\widehat y_g-\widehat X\theta_g^\star=\xi_g,
\qquad
\widehat y_g-\widehat X\widehat\theta_g
=
\widehat y_g-\widehat X(\theta_g^\star+\Delta)
=
\xi_g-\widehat X\Delta.
\]
Hence, the optimality inequality becomes
\[
\frac{1}{2M}\|\xi_g-\widehat X\Delta\|_2^2+\lambda\|\widehat\theta_{g,-0}\|_1
\le
\frac{1}{2M}\|\xi_g\|_2^2+\lambda\|\theta_{g,-0}^\star\|_1.
\]
Expanding the squared norm gives
\[
\|\xi_g-\widehat X\Delta\|_2^2
=
\|\xi_g\|_2^2-2\xi_g^\top\widehat X\Delta+\|\widehat X\Delta\|_2^2.
\]
Substituting this expression into the preceding inequality and canceling the common term $\|\xi_g\|_2^2/(2M)$ from both sides yields
\[
\frac{1}{2M}\|\widehat X\Delta\|_2^2-\frac{1}{M}\xi_g^\top\widehat X\Delta+\lambda\|\widehat\theta_{g,-0}\|_1
\le
\lambda\|\theta_{g,-0}^\star\|_1.
\]
Rearranging gives the basic inequality
\begin{equation}
\frac{1}{2M}\|\widehat X\Delta\|_2^2
\le
\frac{1}{M}\xi_g^\top\widehat X\Delta
+\lambda\left(\|\theta_{g,-0}^\star\|_1-\|\widehat\theta_{g,-0}\|_1\right).
\label{eq:app-basic-inequality}
\end{equation}

By H\"older's inequality and the score condition \eqref{eq:score-condition},
\begin{equation}
\frac{1}{M}\xi_g^\top\widehat X\Delta
=
\left(\frac{1}{M}\widehat X^\top\xi_g\right)^\top\Delta
\le
\left\|\frac{1}{M}\widehat X^\top\xi_g\right\|_\infty\|\Delta\|_1
\le
\frac{\lambda}{2}\|\Delta\|_1.
\label{eq:app-score-bound}
\end{equation}
Since $\|\Delta\|_1=|\Delta_0|+\|\Delta_{S_g}\|_1+\|\Delta_{S_g^c}\|_1$, it remains to control the difference of the $\ell_1$ penalties.

Because $S_g=\operatorname{supp}(A_{g\cdot}^\star)$, the true regulatory coefficients vanish on $S_g^c$. Hence
\[
\|\widehat\theta_{g,-0}\|_1
=
\|\theta_{g,S_g}^\star+\Delta_{S_g}\|_1+\|\Delta_{S_g^c}\|_1
\ge
\|\theta_{g,S_g}^\star\|_1-\|\Delta_{S_g}\|_1+\|\Delta_{S_g^c}\|_1,
\]
where the inequality follows from the reverse triangle inequality. Since $\|\theta_{g,-0}^\star\|_1=\|\theta_{g,S_g}^\star\|_1$,
\begin{equation}
\|\theta_{g,-0}^\star\|_1-\|\widehat\theta_{g,-0}\|_1
\le
\|\Delta_{S_g}\|_1-\|\Delta_{S_g^c}\|_1.
\label{eq:app-penalty-bound}
\end{equation}

Substituting \eqref{eq:app-score-bound} and \eqref{eq:app-penalty-bound} into \eqref{eq:app-basic-inequality} yields
\begin{equation}
\frac{1}{2M}\|\widehat X\Delta\|_2^2+\frac{\lambda}{2}\|\Delta_{S_g^c}\|_1
\le
\frac{\lambda}{2}|\Delta_0|+\frac{3\lambda}{2}\|\Delta_{S_g}\|_1.
\label{eq:app-cone-inequality}
\end{equation}
Dropping the nonnegative quadratic term shows that
\[
\|\Delta_{S_g^c}\|_1\le 3\|\Delta_{S_g}\|_1+|\Delta_0|,
\]
so $\Delta$ belongs to the cone \eqref{eq:rsc-cone}. The RSC condition \eqref{eq:rsc} can therefore be applied to $\Delta$. Dropping the nonnegative $\ell_1$ term in \eqref{eq:app-cone-inequality} gives
\begin{equation}
\frac{\kappa}{2}\|\Delta\|_2^2
\le
\frac{\lambda}{2}\left(|\Delta_0|+3\|\Delta_{S_g}\|_1\right).
\label{eq:app-rsc-step}
\end{equation}

Since $|S_g|\le d$, Cauchy--Schwarz gives $\|\Delta_{S_g}\|_1\le\sqrt{d}\,\|\Delta_{S_g}\|_2$. Applying Cauchy--Schwarz once more,
\begin{align*}
|\Delta_0|+3\|\Delta_{S_g}\|_1
&\le |\Delta_0|+3\sqrt{d}\,\|\Delta_{S_g}\|_2\\
&\le \sqrt{9d+1}\left(|\Delta_0|^2+\|\Delta_{S_g}\|_2^2\right)^{1/2}\\
&\le \sqrt{9d+1}\,\|\Delta\|_2.
\end{align*}
Combining this bound with \eqref{eq:app-rsc-step} yields
\[
\frac{\kappa}{2}\|\Delta\|_2^2
\le
\frac{\lambda}{2}\sqrt{9d+1}\,\|\Delta\|_2.
\]
If $\Delta=0$, the result is immediate. Otherwise, dividing by $\|\Delta\|_2/2$ gives
\[
\|\widehat\theta_g-\theta_g^\star\|_2=\|\Delta\|_2
\le
\frac{\sqrt{9d+1}}{\kappa}\lambda,
\]
which proves \eqref{eq:recovery-bound}.
\end{proof}

\section*{AI Assistance and Code Availability}

The initial version of this manuscript was written entirely by the authors. GPT-5.5 was subsequently used to assist with language editing and stylistic refinement. All AI-assisted revisions were carefully reviewed, verified, and, where necessary, further edited by both authors. The authors take full responsibility for the accuracy, originality, and final content of the manuscript.

The code supporting the numerical experiments and data analyses presented in this work is publicly available at \url{https://github.com/lingqime/grn-rna-velocity}.

\bibliographystyle{unsrt}
\bibliography{references}
\end{document}